\pdfoutput=1

\documentclass{amsart}
\usepackage[paper,pkg/amsaddr=true,pkg/tikz=true,pkg/biblatex/opt={+backref=true}]{zbs}

\newbibmacro*{list:in:delim}[1]{%
  \ifnumgreater{\value{listcount}}{\value{liststart}}
    {\ifboolexpr{
       test {\ifnumless{\value{listcount}}{\value{liststop}}}
       or
       test \ifmoreitems
     }
       {\printdelim{multilistdelim}}
       {\lbx@finallistdelim{#1}}}
    {\bibstring{in}\addspace}}

\DeclareListFormat{language}{%
  \usebibmacro{list:in:delim}{%
    \ifbibstring{#1}
      {\bibxstring{#1}}
      {\ifbibstring{lang#1}
         {\bibxstring{lang#1}}
         {#1}}}%
  \ifbibstring{#1}
    {\bibstring{#1}}
    {\ifbibstring{lang#1}
       {\bibstring{lang#1}}
       {#1}}%
  \usebibmacro{list:andothers}}

\makeatletter 
\patchcmd\caption@subtypehook{\let\label\subcaption@label}
{\let\label\subcaption@label\let\ltx@label\subcaption@label}{}{\fail}
\makeatother

\usetikzlibrary{intersections}
\tikzset{dot/.style={circle,fill=black,inner sep=0.5mm},
every label/.style={inner sep=0pt}}

\makeatletter
\renewcommand{\p@subfigure}{\thefigure\,}
\makeatother

\newlist{cond}{enumerate}{1}
\setlist[cond]{label=(\roman*)}
\crefname{condi}{condition}{conditions}

\usepackage{orcidlink}
\newcommand{\orcid}[1]{\texorpdfstring{\,\orcidlink{#1}}{}}

\let\tqs\textquotesingle
\let\rsqa\rightsquigarrow
\newcommand{\ui}{i^-}
\newcommand{\uj}{j^-}
\newcommand{\calf}{\mathcal{F}}
\newcommand{\mT}{\mathcal{T}}

\DeclareMathOperator{\KG}{KG}
\DeclareMathOperator{\cd}{cd}
\DeclareMathOperator{\conv}{conv}

\DeclareMathOperator{\vol}{vol}
\newcommand{\Assoc}{\mathrm{Assoc}}
\newcommand{\Diag}{\mathrm{Diag}}
\newcommand{\Perm}{\mathrm{Perm}}

\title{A Lovász-Kneser theorem for triangulations}
\author[Anton Molnar]{Anton Molnar\nfts{1}\orcid{0009-0003-2824-8382}}
\author[Cosmin Pohoata]{Cosmin Pohoata\nfts{1}\orcid{0009-0004-4842-2939}}
\author[Michael Zheng]{Michael Zheng\nfts{1}\orcid{0009-0001-2406-6130}}

\address{\nfts{1}Department of Mathematics, Emory University, Atlanta, GA 30322, USA}
\email{\{anton.molnar, cosmin.pohoata, xiangxiang.michael.zheng\}@emory.edu}

\author[Daniel G.\ Zhu]{Daniel G.\ Zhu\nfts{2}\orcid{0000-0001-5675-8353}}
\address{\nfts{2}Department of Mathematics, Princeton University, Princeton, NJ 08544, USA}
\email{zhd@princeton.edu}

\begin{document}
\begin{abstract}
    We show that the Kneser graph of triangulations of a convex $n$-gon has chromatic number $n-2$.
\end{abstract}

\maketitle

\section{Introduction} \label{sec:intro}
For a finite set system $\calf$, define the \vocab{Kneser graph} $\KG(\calf)$ to be the graph whose vertices are the elements of $\calf$ and two vertices $F$ and $F'$ are connected by an edge if $F\cap F' = \emptyset$. In 1955, Kneser \cite{Kneser} considered the chromatic number of $\KG(\binom{[n]}{k})$, where $\binom{[n]}{k}$ is the family of $k$-element subsets of $[n] = \set{1,2,\ldots, n}$. Specifically, Kneser found a simple coloring with $\max(n-2k+2, 1)$ colors, and conjectured that it is optimal. In 1978, Lov\'{a}sz \cite{Lovasz78} proved Kneser's conjecture using the Borsuk-Ulam theorem \cite{BU}, kickstarting the study of topological methods in combinatorics. Although other proofs have since been found, they all have a distinctly topological flavor. To name a few, B\'ar\'any \cite{Barany} deduced the Lov\'asz-Kneser theorem from the Lyusternik-Schnirelmann theorem \cite{LS} (now considered a standard corollary of the Borsuk-Ulam theorem) and Gale's lemma \cite{Gale}. A simplification of this proof removing the need for Gale's lemma was found by Greene \cite{Greene02}. A proof using the topological Radon theorem \cite{Radon} (another corollary of the Borsuk-Ulam theorem) was given by Frick \cite{Frick}, while a purely combinatorial proof based on the Tucker lemma \cite{Tucker} (a combinatorial analogue of the Borsuk-Ulam theorem) was given by Matou\v{s}ek \cite{Matousek}.  

In this paper, we consider the chromatic number of $\KG(\mT_n)$, where $\mT_n$ denotes the set of triangulations of a convex $n$-gon. Here, we represent a triangulation as the set of its diagonals, so that two triangulations are adjacent in the Kneser graph if they share no diagonals. We further represent a diagonal as the set of its endpoints, and identify the vertices of the $n$-gon with the set $[n]$. Thus, every triangulation $T \in \mT_n$ is a subset of the set
\[\Diag_n \coloneq \setmid*{\set{i, j} \in \binom{[n]}{2}}{i-j \not\equiv \pm1 \pmod{n}}.\]
of all diagonals of our convex $n$-gon.

Our main result is the following:
\begin{thm} \label{thm:main}
For every $n \geq 3$,
\[\chi(\KG(\mT_n)) = n-2.\]
\end{thm}
In other words, one cannot partition $\mT_n$ into fewer than $n-2$ intersecting families in which any two triangulations share a diagonal.

It is straightforward to show that $\chi(\KG(\mT_n)) \leq n-2$. Indeed, we may color a triangulation $T$ with the unique color $i \in [n-2]$ such that $T$ contains the triangle $\set{i, n-1, n}$. Alternatively for $n \geq 4$, we may $T \in \mT_n$ with any $i \in [n-2]$ such that $T$ contains the triangle $\set{i,i+1,i+2}$. This is well-defined since any triangulation $T$ consists of $n-2$ triangles, so since there are $n$ sides of the polygon there must be at least two triangles of the form $\set{i,i+1,i+2}$ for $i \in [n]$ (addition is taken modulo $n$). Moreover, since triangles $\set{n-1,n,1}$ and $\set{n,1,2}$ intersect and thus cannot both be in $T$, we can choose some $i \in [n-2]$.

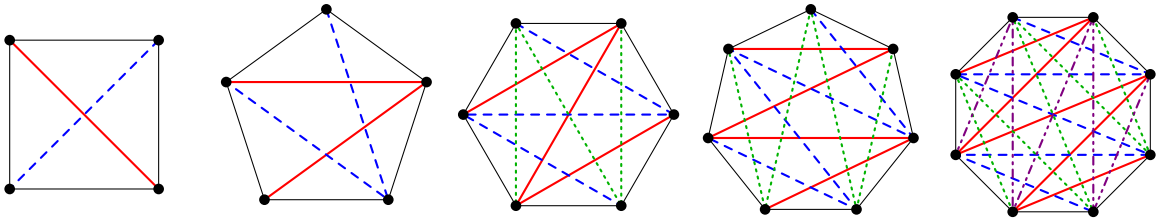
\begin{figure}[t] 
\centering
\begin{tikzpicture}[scale=1.4,s0/.style={thick,line cap=round,line join=round},s1/.style={s0,red},s2/.style={s0,blue,dashed},s3/.style={s0,green!70!black,dotted},s4/.style={s0,violet,dash dot dot}]
\coordinate (a1) at (135:1);
\coordinate (a2) at (45:1);
\coordinate (a3) at (-45:1);
\coordinate (a4) at (-135:1);
\draw[s1] (a1)--(a3);
\draw[s2] (a2)--(a4);
\draw (a1) node[dot]{} --(a2) node[dot]{} --(a3) node[dot]{} --(a4) node[dot]{} --cycle;
\begin{scope}[xshift=2.3cm]
\coordinate (b1) at (162:1);
\coordinate (b2) at (90:1);
\coordinate (b3) at (18:1);
\coordinate (b4) at (-54:1);
\coordinate (b5) at (-126:1);
\draw[s1] (b1)--(b3)--(b5);
\draw[s2] (b2)--(b4)--(b1);
\draw (b1) node[dot]{} --(b2) node[dot]{} --(b3) node[dot]{} --(b4) node[dot]{} --(b5) node[dot]{} --cycle;
\end{scope}
\begin{scope}[xshift=4.6cm]
\coordinate (c1) at (180:1);
\coordinate (c2) at (120:1);
\coordinate (c3) at (60:1);
\coordinate (c4) at (0:1);
\coordinate (c5) at (-60:1);
\coordinate (c6) at (-120:1);
\draw[s1] (c1)--(c3)--(c6)--(c4);
\draw[s2] (c2)--(c4)--(c1)--(c5);
\draw[s3] (c3)--(c5)--(c2)--(c6);
\draw (c1) node[dot]{} --(c2) node[dot]{} --(c3) node[dot]{}
--(c4) node[dot]{} --(c5) node[dot]{} --(c6) node[dot]{} --cycle;
\end{scope}
\begin{scope}[xshift=6.9cm]
\coordinate (d1) at (990/7:1);
\coordinate (d2) at (90:1);
\coordinate (d3) at (270/7:1);
\coordinate (d4) at (-90/7:1);
\coordinate (d5) at (-450/7:1);
\coordinate (d6) at (-810/7:1);
\coordinate (d7) at (-1170/7:1);
\draw[s1] (d1)--(d3)--(d7)--(d4)--(d6);
\draw[s2] (d2)--(d4)--(d1)--(d5)--(d7);
\draw[s3] (d3)--(d5)--(d2)--(d6)--(d1);
\draw (d1) node[dot]{} --(d2) node[dot]{} --(d3) node[dot]{} --(d4) node[dot]{}
--(d5) node[dot]{} --(d6) node[dot]{} --(d7) node[dot]{} --cycle;
\end{scope}
\begin{scope}[xshift=9.2cm]
\coordinate (e1) at (157.5:1);
\coordinate (e2) at (112.5:1);
\coordinate (e3) at (67.5:1);
\coordinate (e4) at (22.5:1);
\coordinate (e5) at (-22.5:1);
\coordinate (e6) at (-67.5:1);
\coordinate (e7) at (-112.5:1);
\coordinate (e8) at (-157.5:1);
\draw[s1] (e1)--(e3)--(e8)--(e4)--(e7)--(e5);
\draw[s2] (e2)--(e4)--(e1)--(e5)--(e8)--(e6);
\draw[s3] (e3)--(e5)--(e2)--(e6)--(e1)--(e7);
\draw[s4] (e4)--(e6)--(e3)--(e7)--(e2)--(e8);
\draw (e1) node[dot]{} --(e2) node[dot]{} --(e3) node[dot]{} --(e4) node[dot]{}
--(e5) node[dot]{} --(e6) node[dot]{} --(e7) node[dot]{} --(e8) node[dot]{} --cycle;
\end{scope}
\end{tikzpicture}
\caption{$\floor{n/2}$ pairwise disjoint zig-zag triangulations of an $n$-gon for $4 \leq n \leq 8$.} \label{fig:zig-zag}
\end{figure}

On the other hand, it is also not difficult to see that $\KG(\mT_n)$ must always contain a clique of size $\floor{n/2}$, which implies that $\chi(\KG(\mT_n)) \geq \floor{n/2}$. Such a clique can be obtained by taking $\floor{n/2}$ ``zig-zag'' triangulations of the form $\set{i,i+2}, \set{i+2,i-1}, \set{i-1,i+3}, \set{i+3,i-2}, \ldots$ for $i \in [\floor{n/2}]$, as shown in \cref{fig:zig-zag}. Note that for $n$ even, this construction is essentially Walecki's Hamiltonian decomposition of $K_{n+1}$ (see e.g.\ \cite{MR2394738}). In fact, $\floor{n/2}$ is the clique number of $\KG(\mT_n)$, as each triangulation in $\mT_n$ consists of $n-3$ diagonals and $\abs{\Diag_n} = n(n-3)/2$.

The main difficulty behind \cref{thm:main} thus lies in establishing that $\chi(\KG(\mT_n)) \geq n-2$ for all $n \geq 3$. While there exist results bounding $\chi(\KG(\calf))$ for arbitrary set systems $\calf$, applying these results does not work for $\mT_n$ (see \cref{sec:motiv}), meaning that we have to use the structure of $\mT_n$ in a deeper way.

\begin{figure}[t]
\centering
\includegraphics[height = 8cm]{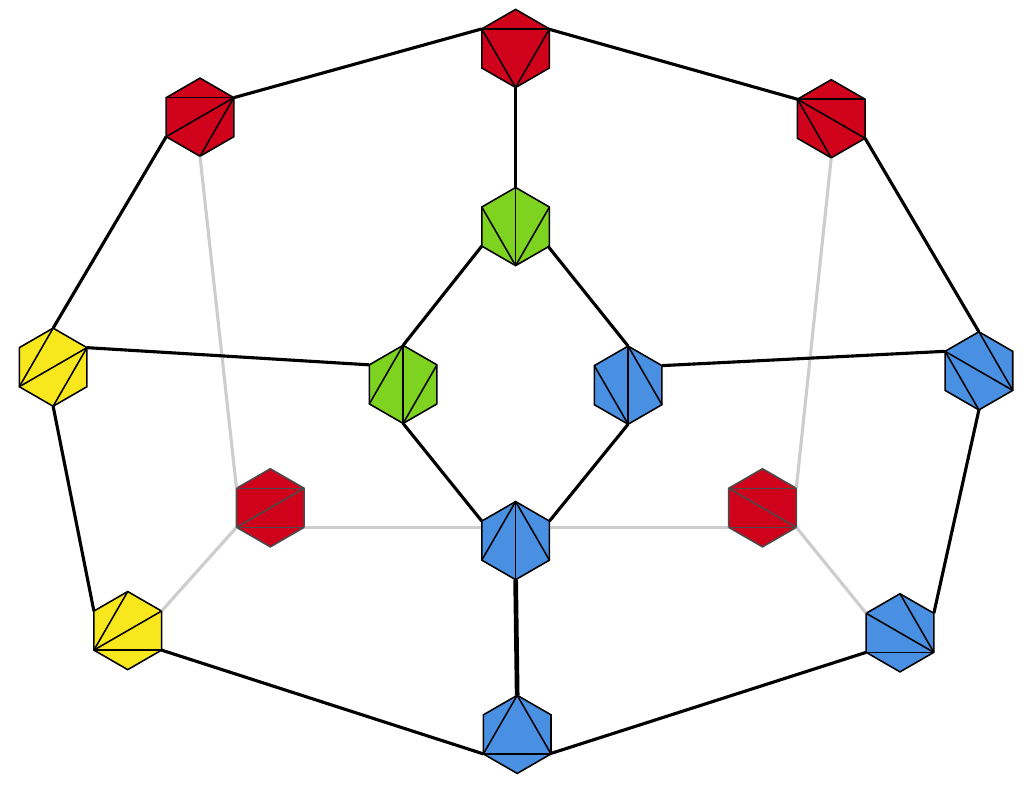}
\caption{The associahedron $\Assoc_3$ and a proper $4$-coloring of $\KG(\mT_6)$.} \label{fig:assoc}
\end{figure}
To do so, we work with the associahedron $\Assoc_{n-3}$, an $(n-3)$-dimensional combinatorial polytope discovered by Tamari \cite{tamari} and Stasheff \cite{MR158400} whose $k$-dimensional faces correspond to collections of $n-3-k$ noncrossing diagonals of a convex $n$-gon, with inclusion reversing under this correspondence. In particular, triangulations correspond to vertices of $\Assoc_{n-3}$, diagonals correspond to facets of $\Assoc_{n-3}$, and a triangulation contains a diagonal if and only if the respective vertex is contained in the respective facet. A crucial consequence is that the associahedron encodes all the data needed to reconstruct $\KG(\mT_n)$. Specifically, if we define the Kneser graph $\KG(P)$ of a combinatorial polytope $P$ to be the graph with vertices given by the vertices of $P$, with two vertices adjacent in $\KG(P)$ if they do not share a facet in $P$, then we have $\KG(\Assoc_{n-3}) = \KG(\mT_n)$. The case $n=6$ is illustrated in \cref{fig:assoc}.

After this reformulation, we derive \cref{thm:main} from general results showing that for $d$-dimensional polytopes $P$ satisfying certain conditions, we have $\chi(\KG(P)) \geq d+1$. The proofs of these results use topology in the form of the Lyusternik-Schnirelmann theorem, and are quite similar in spirit to B\'ar\'any's proof \cite{Barany} of the original Lov\'asz-Kneser theorem.

\subsection*{Stability results} In 1978, Schrijver \cite{Schrijver} strengthened the Lov\'asz-Kneser theorem by finding an induced subgraph of $\KG(\binom{[n]}{k})$ that is \vocab{color-critical}, i.e.\ having the same chromatic number as $\KG(\binom{[n]}{k})$ and the property that deleting any vertex would decrease the chromatic number. Analogously, we find a small induced subgraph of $\KG(\mT_n)$ with the same chromatic number.
\begin{thm} \label{thm:stab}
Let $F_k$ be the $k$th Fibonacci number (indexed so that $F_1 = F_2 = 1$). Then for each $n \geq 3$ there exists a subset $\mT^{(3)}_n \subseteq \mT_n$ (to be defined in \cref{sec:stab}) such that $\abs{\mT^{(3)}_n} = F_{2n-5}$ and 
\[\chi(\KG(\mT^{(3)}_n)) = n-2.\]
\end{thm}
It is well-known that $\abs{\mT_n} = \Theta(4^n/n^{3/2})$ and $F_{2n-5} = \Theta((\frac{3+\sqrt{5}}{2})^n) = \Theta((2.618\ldots)^n)$, so $\mT^{(3)}_n$ is an exponentially small portion of $\mT_n$. However, in contrast to the work of Schrijver, these graphs are in general not color-critical and computational evidence suggests that it may still be possible to delete substantially more vertices while preserving the chromatic number.

In addition, we show that a crucial role in $\KG(\mT_n)$ is played by the \vocab{star triangulations}, which consist of the $n-3$ diagonals emanating from a given vertex. Specifically, the star triangulations are exactly those whose deletion affects the chromatic number.
\begin{thm} \label{thm:star}
If $n \geq 3$ and $T \in \mT_n$, we have $\chi(\KG(\mT_n \setminus \set{T})) = n-2$ if and only if $T$ is not a star triangulation.
\end{thm}

\subsection*{Outline}
We prove \cref{thm:main} in \cref{sec:mainproof} and \cref{thm:stab,thm:star} in \cref{sec:stab}. We discuss the possibility of higher-dimensional generalizations in \cref{sec:hidim} and finish with some concluding remarks in \cref{sec:final}.

\section{Proof of \texorpdfstring{\cref{thm:main}}{Theorem \ref{thm:main}}} \label{sec:mainproof}
\subsection{How not to prove \texorpdfstring{\cref{thm:main}}{Theorem \ref{thm:main}}} \label{sec:motiv}
There are a number of generalizations of the Lov\'asz-Kneser theorem that apply to arbitrary set systems. Most notably, in 1988 Dol\tqs nikov \cite{Dolnikov} showed a result which is nowadays thought of as a generalization of Greene's proof of the Lov\'asz-Kneser theorem, though this was not Dol\tqs nikov's original proof. To state it, say a set system $\calf \subseteq 2^S$ is \vocab{$2$-colorable} if there exists a $2$-coloring of $S$ such that no element of $\calf$ is monochromatic. The \vocab{$2$-colorability defect} $\cd_2(\calf)$ of $\calf$ is the minimum size of a set $S'$ such that the elements of $\calf$ disjoint from $S'$ form a $2$-colorable family. Dol\tqs nikov's generalization of the Lov\'asz-Kneser theorem can be now stated as follows:
\begin{thm}[Dol\tqs nikov] \label{Dol}
For any finite set system $\calf$, 
\[
\chi(\KG(\calf)) \geq \cd_2(\calf).
\]
\end{thm}

For example, when $n \geq 2k$ are positive integers, it is not difficult to check that $\cd_2(\binom{[n]}{k}) = n-2k+2$, so this result generalizes the Lov\'asz-Kneser theorem.

\begin{figure}[t] 
\centering
\begin{tikzpicture}[scale=1.6,s1/.style={red,very thick},s2/.style={blue},every label/.append style={label distance=1pt}]
\coordinate (c1) at (180:1);
\coordinate (c2) at (120:1);
\coordinate (c3) at (60:1);
\coordinate (c4) at (0:1);
\coordinate (c5) at (-60:1);
\coordinate (c6) at (-120:1);
\draw[s2] (c3)--(c6)--(c2)--(c5)--(c1)--(c4);
\draw[s1] (c2)--(c4)--(c6) (c1)--(c3)--(c5);
\draw (c1) node[dot,label={180:$1$}]{} --(c2) node[dot,label={120:$2$}]{} --(c3) node[dot,label={60:$3$}]{}
--(c4) node[dot,label={0:$4$}]{} --(c5) node[dot,label={-60:$5$}]{} --(c6) node[dot,label={-120:$6$}]{} --cycle;
\end{tikzpicture}
\caption{A $2$-coloring of $\Diag_6$ demonstrating that $\cd_2(\mT_6) = 0$.} \label{fig:cd2-n=6}
\end{figure}
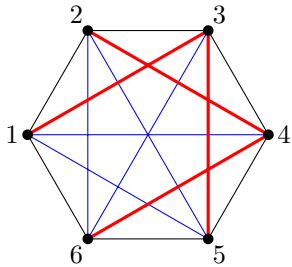

However, \cref{Dol} cannot be used to obtain \emph{any} interesting lower bound for $\chi(\KG(\mT_n))$, as $\cd_2(\mT_n) = 0$ for $n \geq 6$. Indeed, if diagonals of the form $\set{i,i+2}$ for $i \in [n-2]$ are colored red and all other diagonals are colored blue, then no triangulation is monochromatic. By the same argument as in the introduction, every triangulation contains a red diagonal. Moreover, every triangle containing the edge $\set{n,1}$ contains a blue diagonal as a side, so every triangulation must contain a blue diagonal as well. This construction is illustrated in \cref{fig:cd2-n=6}.

The failure of \cref{Dol} to yield useful results can perhaps be explained by its proofs (e.g., see \cite[Thm.~3.4.1]{MR1988723}), where elements of the ground set $S$ are arbitrarily assigned points in general position. This is reasonable in the case where $\calf = \binom{[n]}{k}$, as the ground set $[n]$ has no inherent structure. However, in our case, $S$ is $\Diag_n$, and there are important relations between the elements of $\Diag_n$ which \cref{Dol} throws away entirely. This motivates the use of the associahedron, which provides a geometric structure which allows us to choose points more carefully.

\subsection{The Kneser graph of a polytope}
\label{subsec:KG(poly)}
In this section, we will generalize to the case of an arbitrary convex polytope and find various conditions that imply bounds on the chromatic number of its Kneser graph.

Let $P$ be a polytope of full dimension in a real vector space $X$ of dimension $d$. Let $X^*$ denote the dual space of $X$. (To avoid later confusion, we will avoid identifying $X$ with $\setr^d$.) Let $V(P)$ and $F(P)$ denote the vertices and facets of $P$, respectively, and for $f \in F(P)$ let $f^\perp$ denote an outward normal. Finally, we recall the definition of $\KG(P)$ as a graph whose vertices are given by $V(P)$ and where two vertices are connected if and only if no facet contains both.

The following result is our main tool for bounding the chromatic number; the implication $\text{\ref{item:gale}} \implies \text{\ref{item:chi}}$ can be regarded as a generalization of B\'ar\'any's proof \cite{Barany} of the Lovasz-Kneser theorem.
\begin{lem} \label{lem:polybu}
Consider the following conditions:
\begin{cond}
\item \label{item:obtuse} There exists a bilinear form $B$ on $X^*$ such that $B(f_1^\perp, f_2^\perp) > 0$ for any $f_1,f_2\in F(P)$ that share a vertex.
\item \label{item:vectors} There exists an assignment of a vector $x_f \in X$ to every $f \in F(P)$ such that if $f_1$ and $f_2$ share a vertex, $\gen{f_1^\perp, x_{f_2}} > 0$.
\item \label{item:gale} There exists an assignment of a point $u_f \in S^{d-1}$ to every $f \in F(P)$ such that for each open hemisphere $H \subseteq S^{d-1}$ there exists a vertex $v \in V(P)$ such that $u_f \in H$ for every $f \in F(P)$ containing $v$.
\item \label{item:chi} $\chi(\KG(P)) \geq d+1$.
\end{cond}
Then we have $\text{\ref{item:obtuse}} \implies \text{\ref{item:vectors}} \implies \text{\ref{item:gale}} \implies \text{\ref{item:chi}}$.
\end{lem}

Before we prove \cref{lem:polybu}, we first go over a few preliminaries. First is the following basic fact from convex geometry, which follows from the Farkas lemma \cite{Farkas}.
\begin{prop} \label{prop:farkas}
Consider some $w \in X^*$ and let $v$ be a vertex of $P$ that maximizes $\gen{w, v}$. Then $w$ is a \vocab{conic combination} (a linear combination with nonnegative coefficients) of the outward normals of the facets containing $v$.
\end{prop}

We will also use a topological tool in the form of the Lyusternik-Shnirelmann theorem.
\begin{thm}[Lyusternik-Shnirelmann \cite{LS}] \label{thm:ls}
If sets $A_1, A_2, \ldots, A_d$ form an open cover of the $(d-1)$-sphere $S^{d-1}$, then there exists some $i \in [d]$ and $x \in S^{d-1}$ such that $x, -x \in A_i$.
\end{thm}

\begin{rmk}
It follows from \cref{prop:farkas} that the symmetrization of any $B$ that satisfies \cref{item:obtuse} must be positive definite. Indeed, given any nonzero $w \in X^*$, take a vertex $v$ of $P$ maximizing $\gen{w, v}$. By \cref{prop:farkas} we may write $w = \sum_f c_f f^\perp$ where $c_f \geq 0$ and the sum is over facets $f$ adjacent to $v$, which in turn implies that
\[B(w, w) = \sum_{f,f'} c_fc_{f'}B(f^\perp, (f')^\perp) > 0.\]
Therefore, \cref{item:obtuse} has the geometric interpretation that we may impose a Euclidean structure on $X$ such that any two facets sharing a vertex meet at an obtuse angle.
\end{rmk}

\begin{proof}[Proof of \cref{lem:polybu}]
\textbf{$\text{\ref{item:obtuse}} \implies \text{\ref{item:vectors}}$.} Set $x_f = B(f^\perp, -)$.

\textbf{$\text{\ref{item:vectors}} \implies \text{\ref{item:gale}}$.} Suppose $x_f \in X$ satisfy \cref{item:vectors}. Pick an arbitrary inner product on $X$ and without loss of generality assume all the $x_f$ have length $1$. Let $u_f \in S^{d-1}$ correspond to $x_f$ under the identification of $S^{d-1}$ with the unit sphere in $X$.

Under this identification, an open hemisphere is given by $\setmid{x \in S^{d-1}}{\gen{w, x} > 0}$ for some nonzero $w \in X^*$. Let $v$ be a vertex of $P$ that maximizes $\gen{w, v}$ and let $S$ be the set of facets containing $v$. By \cref{prop:farkas} we may write $w = \sum_{f \in S} c_f f^\perp$ for nonnegative $c_f$, which are not all zero since $w$ is nonzero. Then, for each $f \in S$, we must have
\[\gen{w, x_f} = \sum_{f' \in S} c_{f'} \gen{(f')^\perp, x_f} > 0,\]
meaning that $v$ satisfies the desired condition.

\textbf{$\text{\ref{item:gale}} \implies \text{\ref{item:chi}}$.}
Suppose there exists a proper $d$-coloring of $\KG(P)$. Then, for $i \in [d]$, define $A_i \subseteq S^{d-1}$ to be the set of $w \in S^{d-1}$ for which there exists some $v \in V(P)$ of color $i$ such that $\gen{w, u_f} > 0$ for every $f \in F(P)$ containing $v$. Observe that the $A_i$ are open, and that by \cref{item:gale} they cover $S^{d-1}$. By \cref{thm:ls}, there must exist some $i \in [d]$ and nonzero $w \in S^{d-1}$ such that $w, -w\in A_i$. Thus there exist vertices $v, v' \in V(P)$ both of color $i$, such that $\gen{w, u_f} > 0$ for any $f$ containing $v$ and $\gen{w, u_f} < 0$ for any $f$ containing $v'$. However, such $v$ and $v'$ have no facets in common, a contradiction.
\end{proof}

\subsection{The associahedron as a secondary polytope}
So far we have only treated the associahedron as a combinatorial object, but \cref{lem:polybu} concerns polytopes that have been geometrically realized. Our next task is thus to find a suitable realization of the associahedron with which to verify the conditions of \cref{lem:polybu}.

Although there are many realizations of the associahedron (see \cite{MR3437894}), here we will use a realization due to Gelfand, Kapranov, and Zelevinsky \cite{GKZ} which is notable for how it respects the dihedral symmetry of $\mT_n$. Specifically, the associahedron is realized as the secondary polytope of a convex $n$-gon, which can be viewed as the same $n$-gon of which we are taking triangulations. Secondary polytopes have a rich theory, but for simplicity we will only mention those most relevant to our applications; the reader interested in more details can consult \cite{GKZ,MR2743368}.

Let $Q$ be a convex $n$-gon in the plane with vertex $V(Q)$. For a triangulation $T$ of $Q$ (which we identify with a suitable $T \in \mT_n$) and vertex $p \in V(Q)$, let $Q_{T, p} \subseteq Q$ be the region of $Q$ consisting of the union of all triangles in $T$ with $p$ as a vertex. Let $v_T =(\vol Q_{T, p})_{p \in V(Q)} \in \setr^{V(Q)}$, where $\vol R$ denotes the area of $R$. The \vocab{secondary polytope} $\Sigma_Q$ of $Q$ is defined up to translation as the convex hull of the points $\setmid{v_T}{T\in\mT_n}$.

As defined above, $\Sigma_Q$ lies in an $(n-3)$-dimensional affine subspace of $\setr^{V(Q)}$. Indeed, given any affine linear form $\ell$, we find that
\begin{align*}
\sum_{p \in V(Q)} \vol Q_{T, p} \cdot \ell(p) &= \sum_{\triangle p_1p_2p_3 \text { in }T} \vol \triangle p_1p_2p_3 \cdot (\ell(p_1) + \ell(p_2) + \ell(p_3)) \\
&= \sum_{\triangle p_1p_2p_3 \text { in }T} \iint_{\triangle p_1p_2p_3} 3\ell\,\mathrm{d}A \\
&= \iint_{Q} 3\ell\,\mathrm{d}A,
\end{align*}
where $\mathrm{d}A$ is the canonical area form. As a result, we find that after a suitable translation, which we pick arbitrarily and henceforth fix, $\Sigma_Q$ lies in the subspace $X_Q$ of $\setr^{V(Q)}$ consisting of points $(c_p)_{p \in V(Q)}$ such that $\sum_{p \in V(Q)} c_p \cdot (1, p) = 0$ as an equality in $\setr^3$. For the rest of this paper, we will represent elements of $X_Q$ as formal linear combinations of points in $V(Q)$, so that $(c_p)_{p \in V(Q)}$ is written as $\sum_{p \in V(Q)} c_p \cdot p$. It then becomes natural to identify
\[X_Q^* = \set{\text{functions }V(Q) \to \setr}/\set{\text{affine functions }V(Q) \to \setr}.\]
Letting $[\psi]$ denote the class of a function $V(Q) \to \setr$ in $X_Q^*$, the pairing $\gen{-,-} \colon X_Q^* \times X_Q \to \setr$ is given by $\gen{[\psi], \sum_p c_p \cdot p} = \sum_p c_p\psi(p)$.

\begin{rmk} \label{rmk:nocan}
The pragmatic reader might wonder why we are going through so much effort to separate $X_Q$ and $X_Q^*$. The reason is that if we treat $X_Q$ and $X_Q^*$ as a subspace and a quotient of $\smash{\setr^{V(Q)}}$, respectively, it becomes tempting to use the inner product on $\setr^{V(Q)}$, which for our purposes turns out to a complete red herring. This choice of notation emphasizes that elements of $X_Q$ and $X_Q^*$ are fundamentally different objects, which should not be naturally identified.
\end{rmk}

Having described the vertices of $\Sigma_Q$, we now turn to its facets. Given a function $\psi\colon V(Q) \to \setr$, define the function $g_\psi \colon Q \to \setr$ given by
\[g_\psi(x) = \max_{\crampedsubstack{c_p \geq 0 \\ \sum_{p \in V(Q)} c_p = 1 \\ \sum_{p \in V(Q)} c_p\cdot p = x}} \sum_{p \in V(Q)} c_p \psi(p);\]
this should be thought of as the ``upper boundary'' of the convex hull of the points $\setmid{(p, \psi(p)) \in \setr^3}{p \in V(Q)}$. Note that shifting $\psi$ by some affine function shifts $g_\psi$ by the same affine function, so given some class $[\psi] \in X_Q^*$, it makes sense to discuss $g_{\psi}$, which is well-defined up to adding some affine function.

We may now describe how the associahedron is realized as a secondary polytope.
\begin{thm}[Thm.~1.7, Thm.~2.4, Prop.~3.4, Prop.~3.5 in {\cite[Ch.~7]{GKZ}}] \label{thm:GKZ}
$\Sigma_Q$ is a geometric realization of $\Assoc_{n-3}$ such that
\begin{itemize}
    \item for every $T \in \mT_n$, the vertex of $\Assoc_{n-3}$ corresponding to $T$ is realized as $v_T$;
    \item for every $d \in \Diag_n$ splitting $Q$ into polygons $Q_1$ and $Q_2$, the facet of $\Assoc_{n-3}$ corresponding to $d$ is realized as a facet whose outward normals are given by the nonzero $[\psi] \in X_Q^*$ such that $g_\psi$ is affine when restricted to both $Q_1$ and $Q_2$.
\end{itemize}
\end{thm}

\subsection{Putting things together}
We now give two independent proofs that the associahedron satisfies \cref{lem:polybu}, each of which implies \cref{thm:main}.

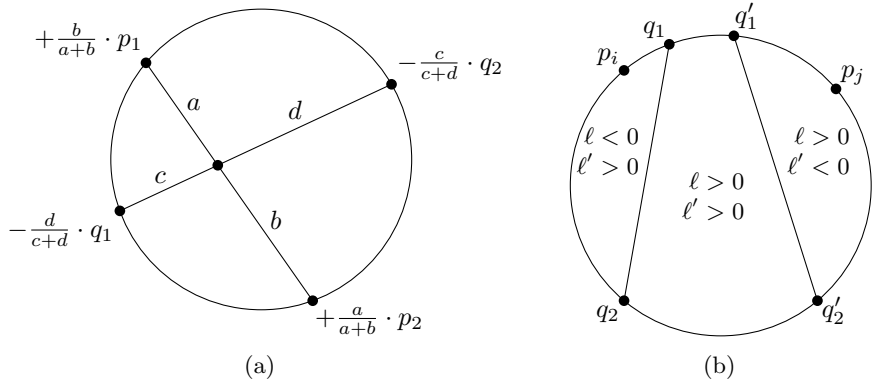
\begin{figure}[t] \centering
\subcaptionbox{\label{fig:p1}}{\begin{tikzpicture}[scale=2,trim left=-2cm,trim right=2cm]
\draw (0,0) circle (1);
\path[name path=d1] (140:1) node[dot,label=120:{$+\frac{b}{a+b} \cdot p_1$}]{} -- (-70:1) node[dot,label=-70:{$+\frac{a}{a+b} \cdot p_2$}]{};
\path[name path=d2] (-160:1) node[dot,label=-160:{$-\frac{d}{c+d} \cdot q_1$}]{} -- (30:1) node[dot,label=30:{$-\frac{c}{c+d} \cdot q_2$}]{};
\draw[name intersections={of=d1 and d2},every label/.append style={label distance=2pt}] (140:1) -- node[coordinate,label=35:$a$]{} (intersection-1) node[dot]{} -- node[coordinate,label=35:$b$]{} (-70:1)
(-160:1) -- node[coordinate,label=115:$c$]{} (intersection-1) -- node[coordinate,label=115:$d$]{} (30:1);
\end{tikzpicture}}\hspace{5pc}%
\subcaptionbox{\label{fig:p2}}{\begin{tikzpicture}[scale=2,trim left=-2cm,trim right=2cm]
\draw (0,0) circle (1);
\draw (110:1) node[dot,label=110:$q_1$]{} -- (-130:1) node[dot,label=-130:$q_2$]{};
\draw (85:1) node[dot,label=85:$q'_1$]{} -- (-50:1) node[dot,label=-50:$q^{\smash{\prime}}_2$]{};
\node[dot,label=130:$p_i$] at (130:1) {};
\node[dot,label=40:$p_j$] at (40:1) {};
\node[font=\small] at (-0.75,0.2) {$\begin{aligned}\ell&<0\\[-4pt]\ell'&>0\end{aligned}$};
\node[font=\small] at (-0.05,-0.1) {$\begin{aligned}\ell&>0\\[-4pt]\ell'&>0\end{aligned}$};
\node[font=\small] at (0.65,0.2) {$\begin{aligned}\ell&>0\\[-4pt]\ell'&<0\end{aligned}$};
\end{tikzpicture}}
\caption{\subref{fig:p1} An element of $X_Q$, used in the proof of \cref{lem:vec}. \subref{fig:p2} Points used in the proof of \cref{lem:circ}.}
\end{figure}

\begin{lem} \label{lem:vec}
$\Sigma_Q$ satisfies \cref{item:vectors} of \cref{lem:polybu}.
\end{lem}
\begin{proof}
Given a facet $f$ of $\Sigma_Q$ corresponding to a diagonal $p_1p_2$ of $Q$, we let $x_f \in X_Q$ be any linear combination $\sum_{p \in V(Q)} c_p \cdot p$ such that $c_p > 0$ if $p = p_1p_2$ and $c_p < 0$ for all other $p$. To show that such an $x_f$ exists, we observe that given any diagonal $q_1q_2$ that crosses $p_1p_2$ at a point $r$, we have
\[\frac{\abs{rp_2}}{\abs{p_1p_2}} \cdot p_1 + \frac{\abs{rp_1}}{\abs{p_1p_2}} \cdot p_2 - \frac{\abs{rq_2}}{\abs{q_1q_2}} \cdot q_1 - \frac{\abs{rq_1}}{\abs{q_1q_2}} \cdot q_2 \in X_Q,\]
where $\abs{-}$ denotes length. This is illustrated in \cref{fig:p1}. Summing over all choices of $q_1q_2$ yields the desired $x_f$.

Now suppose $q_1q_2$ is a diagonal that does not cross $p_1p_2$. If $\ell$ is a nonzero affine function that vanishes at $q_1$ and $q_2$ and is nonnegative at $p_1$ and $p_2$, the function $\psi = \min(\ell, 0)$ is a representative of the outward normal of the facet corresponding to $q_1q_2$. Moreover, $\psi$ is never positive and zero at $p_1$ and $p_2$. It follows that $\gen{[\psi], x_f} > 0$.
\end{proof}
\begin{lem} \label{lem:circ}
If $Q$ is a regular $n$-gon, then $\Sigma_Q$ satisfies \cref{item:obtuse} of \cref{lem:polybu}.
\end{lem}
\begin{proof}
In this proof, all indices are taken modulo $n$.

Let $p_0, p_1, \ldots, p_{n-1}$ be the vertices of $Q$ in order and define the sequence $a_0, a_1, \ldots, a_{n-1}$ given by
\[a_k = \frac{2 - 2\cos(2\pi/n)}{n} + \begin{cases} 2\cos(2\pi/n) & k = 0 \\ -1 & k=\pm 1 \\ 0 & \text{else}.\end{cases}\]
Let $\tilde B$ be the bilinear form on functions $V(Q) \to \setr$ given by
\[\tilde B(\psi, \psi') = \sum_{i,j =0}^{n-1} a_{i-j} \psi(p_i) \psi'(p_j),\]
where indices are taken modulo $n$. To show that this descends to a bilinear form $B$ on $X^*_Q$, it suffices to show that $\tilde B(\psi, -) = 0$ whenever $\psi$ is affine. (Note that since $(a_k)$ is even, $B$ is symmetric.) This is equivalent to the condition
\[\sum_{k=0}^{n-1} a_k = \sum_{k=0}^{n-1} a_k \cos(2\pi k/n) = \sum_{k=0}^{n-1} a_k \sin(2\pi k/n) = 0,\]
which can be easily checked.

We now show that $B$ satisfies the desired condition. Consider two facets $f$ and $f'$ of $\Sigma_Q$ which share a vertex; these correspond to noncrossing diagonals $q_1q_2$ and $q'_1q'_2$ in $Q$. Consider nonzero affine functions $\ell$ and $\ell'$ such that $\ell(q_1) = \ell(q_2) = \ell'(q'_1) = \ell'(q'_2) = 0$ and $\ell(q'_1), \ell(q'_2), \ell'(q_1), \ell'(q_2) \geq 0$. The functions $\psi = \min(\ell, 0)$ and $\psi' = \min(\ell', 0)$ correspond to outward normals to $f$ and $f'$, respectively, so
\[B(f^\perp, (f')^\perp) = \tilde B(\psi, \psi') = \sum_{\substack{0 \leq i,j < n \\ \ell(p_i), \ell'(p_j) < 0}} a_{i-j} \ell(p_i) \ell'(p_j).\]
However, observe that for all $(i,j)$ with $\ell(p_i), \ell'(p_j) < 0$, we must have $i-j \notin \set{0,\pm 1}$ (see \cref{fig:p2}), so $a_{i-j} > 0$. Since at least one such pair $(i,j)$ exists, we conclude that $B(f^\perp, (f')^\perp) > 0$, as desired.
\end{proof}
\begin{rmk}
As $X_Q$ is a subspace of $\setr^{V(Q)}$, it inherits an inner product. However, as alluded to in \cref{rmk:nocan}, attempting to use this inner product to show \cref{item:obtuse} of \cref{lem:polybu} fails, with orthogonal facets appearing for $n=6$ and negative dot products appearing for $n \geq 7$. 
\end{rmk}

\begin{proof}[Proof of \cref{thm:main}]
It was shown in \cref{sec:intro} that $\chi(\KG(\mT_n)) \leq n-2$. For the lower bound, combine either \cref{lem:vec} or \cref{lem:circ} with \cref{lem:polybu} to find a convex $n$-gon $Q$ for which $\chi(\KG(\Sigma_Q)) \geq n-2$. Applying \cref{thm:GKZ} yields that $\chi(\KG(\mT_n)) \geq n-2$, as desired.
\end{proof}

\section{Stability Results} \label{sec:stab}
\subsection{Motivational remarks} \label{sec:stabmotiv}
In this section we will show that the chromatic number of $\KG(\mT_n)$ stays unchanged even if some vertices are deleted. Before we do so in full generality, however, it is useful to consider a simple case, as it illustrates many of the ideas that appear. Specifically, we will sketch a proof that $\chi(\KG(\mT_6 \setminus \set{Z})) = 4$, where $Z$ is the triangulation with diagonals $\set{1,5}, \set{2,4}, \set{2,5}$.

To summarize one of our proofs that $\chi(\KG(\mT_6)) = 4$, \cref{lem:vec} finds $x_d \in X_Q$ for $d \in \Diag_6$ such that for any noncrossing $d, d' \in \Diag_6$, we have $\gen{f_{d}^\perp, x_{d'}} > 0$, where $f_d \in F(\Sigma_Q)$ is the facet corresponding to a diagonal $d$. Then, \cref{prop:farkas} implies that every $w \in X_Q^*$ can be written as a conic combination $\sum_{d \in T} c_d f_d^\perp$ for some $T \in \mT_6$. So, for every nonzero $w \in X_Q^*$ there is some $T \in \mT_6$ such that $\gen{w, x_d} > 0$ for every $d \in T$, which implies the desired result from a topological argument.

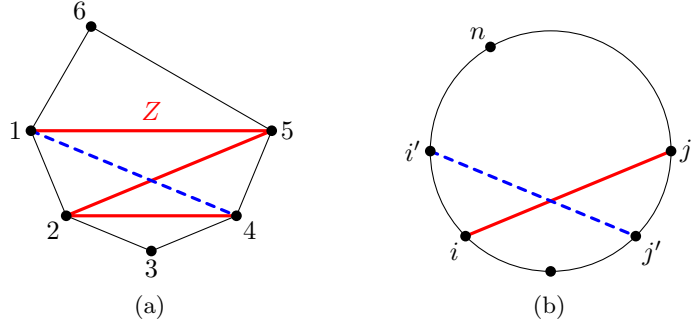
\begin{figure}
\centering
\subcaptionbox{\label{fig:zswap}}{%
\begin{tikzpicture}[scale=1.6,s0/.style={very thick,line cap=round,line join=round},s1/.style={s0,red},s2/.style={s0,blue,dashed},every label/.append style={label distance=1pt},trim left=-1.6cm,trim right=1.6cm]
\coordinate (c1) at (-180:1);
\coordinate (c2) at (-135:1);
\coordinate (c3) at (-90:1);
\coordinate (c4) at (-45:1);
\coordinate (c5) at (0:1);
\coordinate (c6) at (120:1);
\draw[s1] (c1)-- node[above]{$Z$} (c5)--(c2)--(c4);
\draw[s2] (c1)--(c4);
\draw (c1) node[dot,label={-180:$1$}]{} --(c2) node[dot,label={-135:$2$}]{} --(c3) node[dot,label={-90:$3$}]{}
--(c4) node[dot,label={-45:$4$}]{} --(c5) node[dot,label={0:$5$}]{} --(c6) node[dot,label={120:$6$}]{} --cycle;
\end{tikzpicture}}\hspace{5pc}%
\subcaptionbox{\label{fig:gswap}}{%
\begin{tikzpicture}[scale=1.6,s0/.style={very thick,line cap=round,line join=round},s1/.style={s0,red},s2/.style={s0,blue,dashed},every label/.append style={label distance=1pt},trim left=-1.6cm,trim right=1.6cm]
\coordinate (c1) at (-180:1);
\coordinate (c2) at (-135:1);
\coordinate (c3) at (-90:1);
\coordinate (c4) at (-45:1);
\coordinate (c5) at (0:1);
\coordinate (c6) at (120:1);
\draw[s1] (c5)--(c2);
\draw[s2] (c1)--(c4);
\draw (0,0) circle (1);
\draw (c1) node[dot,label={-180:$i'$}]{} (c2) node[dot,label={-135:$i$}]{}(c3) node[dot]{}
(c4) node[dot,label={-45:$j^{\smash\prime}$}]{} (c5) node[dot,label={0:$j$}]{} (c6) node[dot,label={120:$n$}]{};
\end{tikzpicture}}
\caption{\subref{fig:zswap} An illustration of the argument that $\chi(\KG(\mT_6 \setminus \set{Z})) = 4$. \subref{fig:gswap} An illustration of the $\rsqa$ relation. Here $\set{i,j} \rsqa \set{i',j'}$.}
\end{figure}

Examining the proof of this last step, we find that as long as we can always choose $T$ to be not equal to $Z$, we will have $\chi(\KG(\mT_6 \setminus \set{Z})) = 4$. To ensure that this doesn't happen, it suffices to find an alternate argument for the case when $w$ is a conic combination of $f_{1,5}^\perp, f_{2,4}^\perp, f_{2,5}^\perp$. (For clarity we omit curly brackets in subscripts.) Suppose that it were additionally the case that $\gen{f_{2,5}^\perp, x_{1,4}} > 0$. Then we would have $\gen{w, x_{1,4}} > 0$ in addition to $\gen{w, x_{1,5}}, \gen{w, x_{2,4}} > 0$, implying that we can instead choose $T$ to be $\set{\set{1,4}, \set{1,5}, \set{2,4}}$. Indeed, this turns out to be possible, though unlike \cref{lem:vec}, which applies for all choices of $Q$, we need to choose $Q$ carefully. (In particular, $Q$ cannot be a regular hexagon.) This process is illustrated in \cref{fig:zswap}.

The rest of this section is organized as follows. In \cref{sec:swap} we find a collection of relations $\gen{f_{d}^\perp, x_{d'}} > 0$ that can be satisfied for some $Q$, in addition to the requirements from \cref{item:vectors} of \cref{lem:polybu}. This allows us to delete certain triangulations from our graph. In \cref{sec:star} we show that, with one exception, this process allows any non-star triangulation to be deleted, proving \cref{thm:star}. In \cref{sec:fib} we characterize and enumerate a set of triangulations that can be deleted simultaneously, proving \cref{thm:stab}.

\subsection{A swapping relation} \label{sec:swap}
We define a relation generalizing the $\set{2,5} \to \set{1,4}$ swap in the previous section, illustrated in \cref{fig:gswap}.
\begin{defn}
For $\set{i,j}, \set{i',j'} \in \Diag_n$ with $i<j$ and $i'<j'$, write $\set{i,j} \rsqa \set{i',j'}$ if
\begin{itemize}
\item if $\set{i,j}$ and $\set{i',j'}$ do not cross, or
\item $i' < i < j' < j < n$ and $j'-i > 1$.
\end{itemize}
For $T, T' \in \mT_n$, write $T \rsqa T'$ if $d \rsqa d'$ for all $d \in T$ and $d' \in T'$.
\end{defn}

\begin{lem} \label{lem:lac}
There exists a convex $n$-gon $Q$ and $x_d \in X_Q$ for every $d \in \Diag_n$ such that if $d \rsqa d'$, then $\gen{f_d^\perp, x_{d'}} > 0$, where $f_d$ is the facet of $\Sigma_Q$ corresponding to $d$.
\end{lem}
\begin{proof}
Observe that by adding a small multiple of the $x_d$ constructed in \cref{lem:vec}, we may relax the condition to $\gen{f_d^\perp, x_{d'}} \geq 0$ in the case where $d$ and $d'$ don't cross.

Let $y_n = 1$ and consider real numbers $1 < y_1 < y_2 < \cdots < y_{n-1}$, where at the end of the proof we will take $y_1$ to be large, $y_2$ to be sufficiently large in terms of $y_1$, and so on until $y_{n-1}$ is sufficiently large in terms of $y_1,\ldots,y_{n-2}$. Call this limiting process the \vocab{lacunary limit}.

We choose $Q$ such that the $i$th vertex $p_i$ lies at $(y_i, y_i^2)$; as all the points lie on a parabola, the points $p_1,p_2,\ldots,p_n$ indeed form the vertices of a convex $n$-gon in that order. For $i \in [n-1]$, let
\[\ui = \begin{cases} i-1 & i > 1 \\ n & i = 1.\end{cases}\]
Then, for $2 \leq i \leq n-1$ take $x_{i,n} = 0$ and for $i < j < n$ with $j-i > 1$ take
\[x_{i,j} = \frac{\abs{rp_j}}{\abs{p_ip_j}} \cdot p_i + \frac{\abs{rp_i}}{\abs{p_ip_j}} \cdot p_j - \frac{\abs{rp_{\uj}}}{\abs{p_{\ui}p_{\uj}}} \cdot p_{\ui} - \frac{\abs{rp_{\ui}}}{\abs{p_{\ui}p_{\uj}}} \cdot p_{\uj} \in X_Q,\]
where $r$ is the intersection of the line segments $p_ip_j$ and $p_{\ui}p_{\uj}$. By the same argument as in \cref{lem:vec}, we have $\gen{f_d^\perp, x_{d'}} \geq 0$ if $d$ and $d'$ don't cross, so it suffices to show that in the lacunary limit, we have $\gen{f_{i,j}^\perp, x_{i',j'}} > 0$ if $i' < i < j' < j < n$ and $j'-i > 1$.

Fix such a choice of $i,j,i',j'$. We represent $f_{i,j}^\perp$ as $\psi = \min(\ell, 0)$, where $\ell$ is an affine function that is zero at $p_i$ and $p_j$ and negative at $p_{j'}$. In the limit $y_j \to \infty$, line $p_ip_j$ approaches a vertical line, so after appropriate scaling $\psi$ approaches the function $\psi_\infty(u, v) = \min(y_i - u, 0)$. Thus the condition $\gen{f_{i,j}^\perp, x_{i',j'}} > 0$ becomes
\[\gen{[\psi_\infty], x_{i',j'}} = \frac{\abs{rp_{i'}}}{\abs{p_{i'}p_{j'}}} (y_i - y_{j'}) - \frac{\abs{rp_{(i')^-}}}{\abs{p_{(i')^-}p_{(j')^-}}} (y_i - y_{(j')^-}) > 0,\]
where $r$ is the intersection point of $p_{i'}p_{j'}$ and $p_{(i')^-}p_{(j')^-}$. Now taking the limit $y_{j'}\to\infty$, we see that $\abs{rp_{i'}}$ and $\abs{rp_{(i')^-}}$ approach positive limits. Moreover, $(y_i-y_{j'})/\abs{p_ip_{j'}}$ approaches $0$, while $(y_i-y_{(j')^-})/\abs{p_{(i')^-}p_{(j')^-}}$ is constant, which is negative as $(j')^- > i$. Thus this is true in the lacunary limit, as desired.
\end{proof}

\begin{lem} \label{lem:swap}
Let $\mT' \subseteq \mT_n$ be such that for every $T \in \mT_n$ there is a $T' \in \mT'$ with $T \rsqa T'$. Then $\chi(\KG(\mT')) = n-2$.
\end{lem}
\begin{proof}
We have $\chi(\KG(\mT')) \leq \chi(\KG(\mT_n)) \leq n-2$, so it suffices to prove a lower bound.

Take $x_d$ as in \cref{lem:lac}. For any nonzero $w \in X_Q^*$, take some vertex $v \in V(\Sigma_Q)$ maximizing $\gen{w, v}$. By \cref{prop:farkas}, if $v$ corresponds to a triangulation $T$ we may write $w$ as a conic combination of $f_d^\perp$ for $d \in T$. Taking $T' \in \mT'$ such that $T \rsqa T'$, we find that $\gen{w, x_d} > 0$ for every $d \in T'$. An argument similar to the proof of $\text{\ref{item:gale}} \implies \text{\ref{item:chi}}$ in \cref{lem:polybu} shows the bound.
\end{proof}

\subsection{Proof of \texorpdfstring{\cref{thm:star}}{Theorem \ref{thm:star}}} \label{sec:star}
We first show that deleting a star necessarily decreases the chromatic number.
\begin{lem} \label{lem:starconstruct}
If $T \in \mT_n$ is a star, then $\KG(\mT_n \setminus \set{T}) \leq n-3$.
\end{lem}
\begin{proof}
The result is trivial for $n = 3$, so assume otherwise. Without loss of generality assume that the central vertex of $T$ is $n$. It suffices to show that every triangulation $T' \neq T$ contains a diagonal of the form $\set{i,i+2}$ for $i \in [n-3]$ as we can then color $T'$ with an $i$ such that $\set{i, i+2} \in T'$ holds. To see this, observe that $T'$ contains $n-2$ triangles and each triangle contains up to two edges of the $n$-gon, meaning that $T'$ contains at least two triangles of the form $\set{i,i+1,i+2}$, where addition is taken modulo $n$. If $T'$ contains no diagonal of the form $\set{i,i+2}$ for $i\in [n-3]$, then $T'$ can only have two such triangles, which must be $\set{n-2,n-1,n}$ and $\set{n,1,2}$. The fact that there are no other such triangles implies that every other triangle in $T'$ contains exactly one edge of the $n$-gon.

We now claim by induction on $i$ that $\set{i,n}\in T'$ for all $2 \leq i \leq n-2$, which will show that $T' = T$, a contradiction. The base case $i=2$ is true since $\set{n,1,2}$ is a triangle in $T'$. Moreover, for a given $2 \leq i < n-2$, there must exist some $j > i$ such that $\set{i,j,n}$ is a triangle in $n$. The only way for this triangle to contain an edge of the $n$-gon is if $j=i+1$, in which case we are done, or if $j = n-1$, which contradicts the fact that $\set{n-2,n-1,n}$ is a triangle in $T'$. This concludes the proof.
\end{proof}

In light of \cref{lem:swap}, we know that $\chi(\KG(\mT_n \setminus \set{T})) = n-2$ for all $T \in \mT_n$ for which there is some $T' \in \mT_n$ with $T \neq T'$ and $T \rsqa T'$. This covers many non-star triangulations, and by allowing $T$ to be rotated or reflected (which doesn't change the chromatic number by symmetry) we can hit almost all non-star triangulations. However, this process still leaves one example that we have to deal with manually: the triangulation $\Delta = \set{\set{1,3},\set{1,5},\set{3,5}} \in \mT_6$ and its rotation $\Delta' = \set{\set{2,4},\set{2,6},\set{4,6}}$.

\begin{lem} \label{lem:swapmost}
Suppose $T \in \mT_n$ is neither a star, $\Delta$, nor $\Delta'$. Then there exists a rotation or reflection of $T$, denoted by $T'$, and some $T'' \in \mT_n$ such that $T' \neq T''$ and $T' \rsqa T''$.
\end{lem}
\begin{proof}
It suffices to find a copy of the triangulation $Z$ from \cref{sec:stabmotiv} in $T$, i.e.\ (not necessarily consecutive) indices $i_1,i_2,i_3,i_4,i_5,i_6 \in [n]$, going around the $n$-gon in that order, such that $\set{i_2,i_3,i_4}$, $\set{i_2,i_4,i_5}$, $\set{i_1, i_2, i_5}$, and $\set{i_1, i_5, i_6}$ are triangles in $T$. If this is the case, we may rotate or reflect $T$ so that $i_1 < i_2 < i_3 < i_4 < i_5 < i_6 = n$, and let $T''$ be the triangulation obtained by swapping $\set{i_2, i_5}$ for $\set{i_1, i_4}$. 

To find this copy, we proceed by induction, with the base cases of $n \leq 5$ being vacuous as all triangulations are stars. Given a triangulation $T \in \mT_n$ for $n \geq 6$, we know there exists a triangle of the form $\set{i,i+1,i+2}$ in $T$; without loss of generality assume that $i = n-1$, so that deleting vertex $n$ yields a triangulation $T^- \in \mT_{n-1}$. If $T^-$ contains a $Z$-copy we are done. Otherwise, by the inductive hypothesis there are a few cases for $T^-$.

\begin{figure}[t] 
\centering
\begin{tikzpicture}[scale=1.6,s0/.style={thick,line cap=round,line join=round},s1/.style={s0,red},s2/.style={s0,blue,dashed},s3/.style={s0,green!70!black,dotted},s4/.style={s0,violet,dash dot dot},every label/.append style={label distance=1pt,font=\small}]
\coordinate (a1) at (180:1);
\coordinate (a2) at (-150:1);
\coordinate (a3) at (-120:1);
\coordinate (a4) at (-60:1);
\coordinate (a5) at (0:1);
\coordinate (a6) at (0,0.5);
\fill[red!25] (a1)--(a2)--(a3)--(a4)--(a5)--(a6)--cycle;
\draw[s1] (a5)--(a1)--(a4)--(a2);
\draw (a3)--(a4)--(a5);
\draw (a1) node[dot,label={180:$1$}]{} --(a2) node[dot,label={-150:$2$}]{} --(a3) node[dot,label={-120:$3$}]{} arc (-120:-60:1) node[dot,label={-60:$i$}]{} arc (-60:0:1) node[dot,label={0:$n-1$}]{} -- (a6) node[dot,label={90:$n$}]{} --cycle;
\begin{scope}[xshift=3cm]
\coordinate (a1) at (180:1);
\coordinate (a2) at (-120:1);
\coordinate (a3) at (-60:1);
\coordinate (a4) at (-30:1);
\coordinate (a5) at (0:1);
\coordinate (a6) at (0,0.5);
\fill[red!25] (a1)--(a2)--(a3)--(a4)--(a5)--(a6)--cycle;
\draw[s1] (a1)--(a5)--(a2)--(a4);
\draw (a1)--(a2)--(a3);
\draw (a1) node[dot,label={180:$1$}]{} arc (-180:-120:1) node[dot,label={-120:$i$}]{} arc (-120:-60:1) node[dot,label={-60:$n-3$}]{} --(a4) node[dot,label={-30:$n-2$}]{} --(a5) node[dot,label={0:$n-1$}]{} --(a6) node[dot,label={90:$n$}]{} --cycle;\end{scope}
\begin{scope}[yshift=-2cm]
\coordinate (a1) at (180:1);
\coordinate (a2) at (-144:1);
\coordinate (a3) at (-108:1);
\coordinate (a4) at (-72:1);
\coordinate (a5) at (-36:1);
\coordinate (a6) at (0:1);
\coordinate (a7) at (0,0.5);
\fill[red!25] (a1)--(a3)--(a4)--(a5)--(a6)--(a7)--cycle;
\draw[s1] (a6)--(a1)--(a5)--(a3);
\draw (a1)--(a3);
\draw (a1) node[dot,label={180:$1$}]{} --(a2) node[dot,label={-144:$2$}]{} --(a3) node[dot,label={-108:$3$}]{} --(a4) node[dot,label={-72:$4$}]{} --(a5) node[dot,label={-36:$5$}]{} --(a6) node[dot,label={0:$6$}]{} --(a7) node[dot,label={90:$7$}]{} --cycle;\end{scope}
\begin{scope}[xshift=3cm,yshift=-2cm]
\coordinate (a1) at (180:1);
\coordinate (a2) at (-144:1);
\coordinate (a3) at (-108:1);
\coordinate (a4) at (-72:1);
\coordinate (a5) at (-36:1);
\coordinate (a6) at (0:1);
\coordinate (a7) at (0,0.5);
\fill[red!25] (a1)--(a2)--(a3)--(a4)--(a6)--(a7)--cycle;
\draw[s1] (a1)--(a6)--(a2)--(a4);
\draw (a4)--(a6);
\draw (a1) node[dot,label={180:$1$}]{} --(a2) node[dot,label={-144:$2$}]{} --(a3) node[dot,label={-108:$3$}]{} --(a4) node[dot,label={-72:$4$}]{} --(a5) node[dot,label={-36:$5$}]{} --(a6) node[dot,label={0:$6$}]{} --(a7) node[dot,label={90:$7$}]{} --cycle;\end{scope}
\end{tikzpicture}
\caption{The cases in the proof of \cref{lem:swapmost}, with the copy of $Z$ shaded in red.} \label{fig:z}
\end{figure}
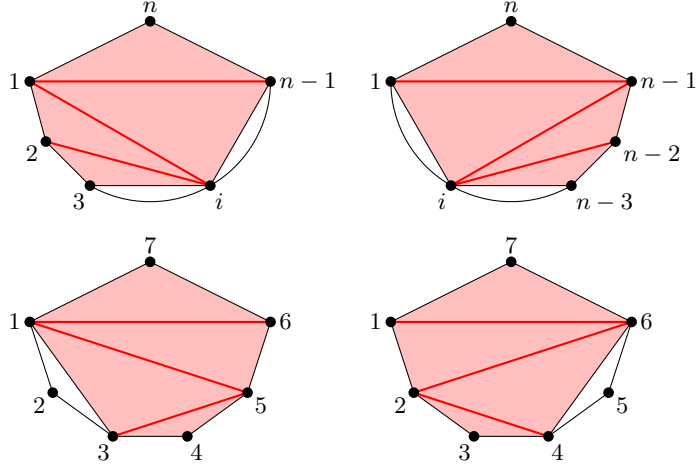
If $T^-$ is a star, then since $T$ is not a star the central vertex $i$ of $T^-$ cannot be $1$ or $n-1$. If $i \geq 4$, then vertices $n-1,i,3,2,1,n$ form a $Z$. If $i \leq n-4$, then vertices $1,i,n-3,n-2,n-1,n$ form a $Z$. This covers all cases except when $n = 6$ and $i = 3$, in which case we have $T = \Delta$, a contradiction. If $T^- = \Delta$, then $6,5,4,3,1,7$ form a $Z$. If $T^- = \Delta'$, then $1,2,3,4,6,7$ form a $Z$. These cases are shown in \cref{fig:z}. Having exhausted all cases, we are done.
\end{proof}

Finally, we deal with $\Delta$ and $\Delta'$. While this is a finite statement and can even be proven with brute force, we present a topological argument that may be useful for inspiring generalizations of \cref{lem:swap}.

\begin{prop} \label{prop:delta}
$\chi(\KG(\mT_6 \setminus \set{\Delta})) = 4$. \label{prop:all_ears}
\end{prop}
\begin{proof}
We have $\chi(\KG(\mT_6\setminus \set{\Delta})) \leq \chi(\KG(\mT_6)) \leq 4$, so it suffices to prove a lower bound.

Let $Q$ be a regular hexagon with vertices $p_1,p_2,\ldots,p_6$. As in the proof of \cref{lem:swap}, it suffices to find $x_d \in X_Q$ for $d \in \Diag_6$ such that for every nonzero $w \in X_Q^*$, there is some $T \in \mT_6$ not equal to $\Delta$ such that $\gen{w, x_d} > 0$ for every $d \in T$.

For $i \in \set{1,3,5}$, define
\begin{align*}
    x_{i,i+3} &= 5p_i - 3p_{i+1} - p_{i+2} + 3p_{i+3} - p_{i+4} - 3p_{i+5} \\
    &= 6 \left(\frac{3}{4} p_i + \frac{1}{4} p_{i+3} - \frac{1}{2} p_{i+1} - \frac{1}{2} p_{i+5}\right)
    + 2\paren*{ \frac{1}{4} p_i + \frac{3}{4} p_{i+3} - \frac{1}{2} p_{i+2} - \frac{1}{2} p_{i+4}} \in X_Q.
\end{align*}
(indices taken modulo $6$) and choose the other $x_{d}$ according to \cref{lem:vec}. Observe that this choice is a special case of the construction in \cref{lem:vec}, so we still have $\gen{f_d^\perp, x_{d'}} > 0$ for any $d, d'$ that don't cross.

Now let $w \in X^\ast_Q$ be an arbitrary nonzero vector. By \cref{prop:farkas}, $w$ is a conic combination of $(f_d^\perp)_{d \in T'}$ for some $T' \in \mT_6$. If $T' \neq \Delta$, then we may take $T = T'$. Otherwise, observe that $f_{1,3}^\perp$, $f_{1,5}^\perp$, and $f_{3,5}^\perp$ can be represented as $-\one_2$, $-\one_6$, and $-\one_4$, respectively, where $\one_i$ denotes the indicator function of $p_i$. Thus we may write $w = [-(\lambda_2\one_2+\lambda_4\one_4+\lambda_6\one_6)]$ for some $\lambda_2,\lambda_4,\lambda_6 \geq 0$. Without loss of generality assume $\lambda_4 \leq \lambda_2,\lambda_6$, which implies that
\[\gen{w, x_{1,4}} = 3\lambda_2 + 3\lambda_6 - 3\lambda_4 \geq 3\max(\lambda_2,\lambda_6) > 0.\]
Since $\set{1,3},\set{1,5} \in \Delta$, we may therefore take $T = \set{\set{1,3},\set{1,4},\set{1,5}}$.
\end{proof}

\begin{proof}[Proof of \cref{thm:star}] 
This follows from combining \cref{lem:swap}, \cref{lem:starconstruct}, \cref{lem:swapmost}, and \cref{prop:delta}.
\end{proof}

\begin{rmk}
    Curiously, the colorings defined in  \cref{sec:intro} and \cref{lem:starconstruct} satisfy the stronger condition that in each color class all the triangulations share one triangle. Hence, if one were to define adjacency by not sharing a triangle, \cref{thm:main}, \cref{thm:stab}, and \cref{thm:star} would all still hold.  
\end{rmk}

\subsection{Proof of \texorpdfstring{\cref{thm:stab}}{Theorem \ref{thm:stab}}} \label{sec:fib}
In this section we will construct a set $\mT^{(3)}_n$ satisfying the conditions of \cref{lem:swap}. Unfortunately, this set is a bit difficult to define, so to do so we will switch to a different language.

Define a \vocab{parenthesization} of an ordered list of symbols $\sigma_1, \sigma_2, \ldots, \sigma_m$ to be a way to insert parentheses into the formal product $\sigma_1\sigma_2 \cdots \sigma_n$ such that it can be interpreted as $n-1$ applications of a binary product. For instance, there are $5$ parenthesizations of $\sigma_1, \sigma_2, \sigma_3, \sigma_4$, given by $\sigma_1(\sigma_2(\sigma_3\sigma_4))$, $\sigma_1((\sigma_2\sigma_3)\sigma_4)$, $(\sigma_1 \sigma_2)(\sigma_3 \sigma_4)$, $(\sigma_1(\sigma_2\sigma_3))\sigma_4$, and $((\sigma_1\sigma_2)\sigma_3)\sigma_4$.

We now recursively define the notion of a \vocab{$k$-parenthesization}. A $0$-parenthesization is a parenthesization with a single symbol. For an odd positive integer $k \geq 1$, a $k$-parenthesization is a parenthesization of the form $\pi_1 (\pi_2 \cdots (\pi_{\ell-1}(\pi_{\ell} \sigma)) \cdots )$ for some $\ell \geq 0$, symbol $\sigma$, and $(k-1)$-parenthesizations $\pi_1, \ldots, \pi_\ell$. If $k \geq 2$ is even, a $k$-parenthesization is a parenthesization of the form $(\cdots ((\sigma \pi_1) \pi_2) \cdots \pi_{\ell-1})\pi_\ell$ for some $\ell \geq 0$, symbol $\sigma$, and $(k-1)$-parenthesizations $\pi_1, \ldots, \pi_\ell$.

It is well-known that triangulations in $\mT_n$ are in bijection with parenthesizations of $n-1$ symbols $\sigma_1,\sigma_2,\ldots,\sigma_{n-1}$. One such bijection, which we will use for the remainder of this section, associates a triangulation with the unique parenthesization where, for all triangles $\set{i,j,k}$ in $T$ with $i < j < k$, the parenthesization obtains $\sigma_i \sigma_{i+1} \cdots \sigma_{k-1}$ from multiplying $\sigma_i \sigma_{i+1} \cdots \sigma_{j-1}$ and $\sigma_j \sigma_{j+1} \cdots \sigma_{k-1}$. Let $\smash{\mT^{(k)}_n \subseteq \mT_n}$ be the set of triangulations corresponding to $k$-parenthesizations under the aforementioned bijection.

\begin{figure}
\centering
\begin{tikzpicture}[scale=0.5,every label/.append style={font=\footnotesize}]
\draw[fill=green!30] (1,0)--(19,0) arc (0:180:9 and 7);
\draw[fill=blue!15] (19,0) arc (0:180:1.5) arc (0:180:5) (5,0) arc (0:180:2);
\draw[fill=red!25] (19,0) arc (0:180:1) (15,0) arc (0:180:2.5) arc (0:180:1) (4,0) arc (0:180:1);
\draw (16,0) arc (0:90:6 and 7) (10,7) arc (90:180:5 and 7) (10,7) arc (90:180:4 and 7);	
\draw (4,0) arc (0:180:1.5);
\draw (15,0) arc (0:180:4.5) (15,0) arc (0:180:1) (15,0) arc (0:180:1.5) (15,0) arc (0:180:2) (15,0) arc (0:180:2.5) arc (0:180:2) (8,0) arc (0:180:1);
\draw (19,0) arc (0:180:1);
\draw[thick] (1,0) node[dot,label={-90:$1$}]{} --
(2,0) node[dot,label={-90:$2$}]{} --
(3,0) node[dot,label={-90:$3$}]{} --
(4,0) node[dot,label={-90:$4$}]{} --
(5,0) node[dot,label={-90:$5$}]{} --
(6,0) node[dot,label={-90:$6$}]{} --
(7,0) node[dot,label={-90:$7$}]{} --
(8,0) node[dot,label={-90:$8$}]{} --
(9,0) node[dot,label={-90:$9$}]{} --
(10,0) node[dot,label={-90:$10$}]{} --
(11,0) node[dot,label={-90:$11$}]{} --
(12,0) node[dot,label={-90:$12$}]{} --
(13,0) node[dot,label={-90:$13$}]{} --
(14,0) node[dot,label={-90:$14$}]{} --
(15,0) node[dot,label={-90:$15$}]{} --
(16,0) node[dot,label={-90:$16$}]{} --
(17,0) node[dot,label={-90:$17$}]{} --
(18,0) node[dot,label={-90:$18$}]{} --
(19,0) node[dot,label={-90:$19$}]{} arc (0:90:9 and 7)
node[dot,label={90:$20$}]{} arc (90:180:9 and 7);
\end{tikzpicture}
\caption{An element of $\mT^{(3)}_{20}$. Regions corresponding to $1$-, $2$-, and $3$-parenthesizations are shaded in red, blue, and green, respectively.} \label{fig:t320}
\end{figure}
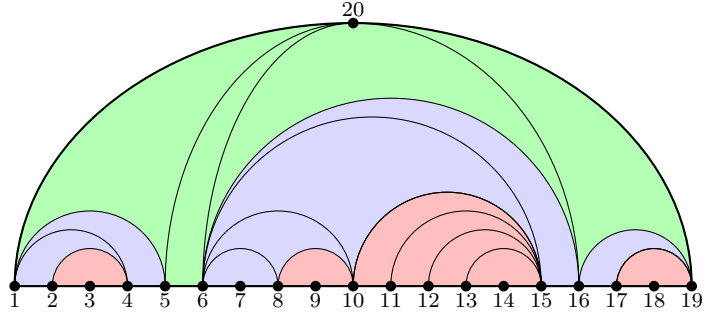
\begin{examp}
An example triangulation in $\mT^{(3)}_{20}$ is shown in \cref{fig:t320}, corresponding to the $3$-parenthesization $\tau_1(\tau_2(\tau_3(\tau_4 \sigma_{19})))$
where the $2$-parenthesizations $\tau_1,\tau_2,\tau_3,\tau_4$ are given by
\begin{align*}
\tau_1 &= (\sigma_1(\sigma_2 \sigma_3))\sigma_4 \\
\tau_2 &= \sigma_5 \\
\tau_3 &= \Bigl(\bigl((\sigma_6\sigma_7) (\sigma_8\sigma_9)\bigr)\bigl(\sigma_{10}(\sigma_{11}(\sigma_{12}(\sigma_{13}\sigma_{14})))\bigr)\Bigr) \sigma_{15} \\
\tau_4 &= \sigma_{16}(\sigma_{17}\sigma_{18}).
\end{align*}
Visually, $1$-parenthesizations correspond to stars (shaded in red) with edges emanating from the rightmost vertex. $2$-parenthesizations correspond to triangulations formed from chains of stars joined by triangles incident to the leftmost vertex (shaded in blue). $3$-parenthesizations correspond to chains of these chains, joined by triangles incident to vertex $n$ (shaded in green).
\end{examp}
\begin{lem} \label{lem:s3}
For every triangulation $T \in \mT_n$ there exists a triangulation $T' \in \mT^{(3)}_n$ such that $T \rsqa T'$.
\end{lem}
\begin{proof}
Let $1 = a_1 < a_2 < \cdots < a_\ell = n-1$ be the vertices connected to $n$ by a side or a diagonal in $T$. Note that $\set{a_i, a_{i+1}}$ must either be a side or a diagonal. For $i \in [\ell-1]$, let $J_i$ consist of the $j\in [a_{i}+2, a_{i+1}-1]$ such that there is a diagonal in $T$ connecting $p_j$ to $p_{j'}$ for some $j' < j$. We may now construct $T'$ as consisting of the following diagonals:
\begin{enumerate}
\item $\set{a_i, n}$ for $2 \leq i \leq \ell-1$,
\item $\set{a_i, a_{i+1}}$ for $i \in [\ell-1]$,
\item $\set{a_i, j}$ for all $i \in [\ell-1]$ and $j \in J_i$, and
\item $\set{j-1, j'}$ for all $i \in [\ell-1]$ and $j \in [a_{i}+2, a_{i+1}-1] \setminus J_i$, where $j'$ is the minimum element of $J_i \cup \set{a_{i+1}}$ greater than $j$.
\end{enumerate}
This determines an element of $\mT^{(3)}_n$, where the $a_i$ determine the breaks between the $2$-parenthesizations and the elements of $J_i$ determine the breaks between the $1$-parenthesizations. For example, if $n = 20$, $(a_1,a_2,a_3,a_4,a_5) = (1,5,6,16,19)$, and $(J_1, J_2, J_3, J_4) = (\set{4}, \emptyset, \set{8,10,15}, \emptyset)$, then $T'$ is the triangulation shown in \cref{fig:t320}.

Every diagonal in $T$ is either in $T'$ by virtue of item (1) or (2) (and thus cannot cross any diagonal in $T'$) or is $\set{j, j'}$ for some $i \in [\ell-1]$ and $\set{j, j'} \neq \set{a_i, a_{i+1}}$ with $a_i \leq j < j' \leq a_{i+1}$. In this latter case, suppose it crosses a diagonal $\set{a_i, k}$ described in item (3). Since $k \in J_i$, we must have $k - j > 1$, but then we indeed have $\set{j,j'} \rsqa \set{a_i, k}$. Moreover, if this diagonal crosses a diagonal $\set{k, k'}$ of type (4), then since $j' \in J_i \cup \set{a_{i+1}}$ we cannot have $j < k < j' < k'$, so we must have $k < j < k' < j'$. This implies $k' \in J_i$, so $k' - j > 1$ and therefore $\set{j,j'} \rsqa \set{k,k'}$.
\end{proof}
\begin{rmk} \label{rmk:nobetter}
In fact, it is not difficult to show that for any $T \in \mT^{(3)}_n$, there does not exist any $T' \neq T$ such that $T \rsqa T'$. Thus \cref{lem:s3} is the best possible result of this form.
\end{rmk}
\begin{prop} \label{prop:count}
$\abs{\mT^{(3)}_n} = F_{2n-5}$.
\end{prop}
\begin{proof}
Let $f_k(t) = t + \sum_{n=3}^\infty \abs{\mT^{(k)}_n} t^{n-1}$ be the generating function where the $t^s$-coefficient is the number of $k$-parenthesizations on $s$ given symbols. Note that $f_0(t) = t$. For $k \geq 1$, note that the presentations of a $k$-parenthesization as  $\pi_1 (\pi_2 \cdots (\pi_{\ell-1}(\pi_{\ell} \sigma)) \cdots )$ or  $(\cdots ((\sigma \pi_1) \pi_2) \cdots \pi_{\ell-1})\pi_\ell$ are unique, so $k$-parenthesizations on $m$ symbols are in bijection with ordered collections of $(k-1)$-parenthesizations with a total of $m-1$ symbols. It follows that
\[f_k(t) = t(1 + f_{k-1}(t) + f_{k-1}(t)^2 + \cdots) = \frac{t}{1 - f_{k-1}(t)}.\]
We may now compute
\[f_1(t) = \frac{t}{1-t}, \quad f_2(t) = \frac{t-t^2}{1-2t}, \quad f_3(t) = \frac{t-2t^2}{1-3t+t^2}.\]
Letting $\varphi = \frac{\sqrt{5}+1}{2}$, one can verify
\[f_3(t) = \frac{\varphi^{-3}/\sqrt{5}}{1 - \varphi^2 t} + \frac{\varphi^{3}/\sqrt{5}}{1 - \varphi^{-2} t} - 2,\]
so
\[
\abs{T^{(3)}_n} = \frac{\varphi^{-3} (\varphi^2)^{n-1} + \varphi^3(\varphi^{-2})^{n-1}}{\sqrt{5}} = \frac{\varphi^{2n-5} - (-\varphi)^{-(2n-5)}}{\sqrt{5}} = F_{2n-5}. \qedhere
\]
\end{proof}

\begin{proof}[Proof of \cref{thm:stab}]
This follows from \cref{lem:swap}, \cref{lem:s3}, and \cref{prop:count}.
\end{proof}

\section{Remarks on Higher-Dimensional Generalizations} \label{sec:hidim}
It is natural to ask whether \cref{thm:main} generalizes to higher-dimensional polytopes and their triangulations, i.e.\ divisions into simplices. However, we will see that there are a number of obstructions to doing so, and in the one case where our techniques easily generalize the results obtained have much more elementary proofs.

\subsection{Higher secondary polytopes}
The secondary polytope $\Sigma_Q$ of a polygon $Q$ was a key ingredient in the proof of \cref{thm:main}, so for $d$-dimensional $Q$ it is natural to look at $\Sigma_Q$ to find the correct generalization of \cref{thm:main}. In this section, we will briefly review some facts about $\Sigma_Q$; a more comprehensive treatment can be found in \cite{GKZ} or \cite{MR2743368}.

A \vocab{subdivision} of $Q$ is a collection $S$ of $d$-dimensional convex polytopes such that $\bigcup_{P \in S} P = Q$, $V(P) \subseteq V(Q)$ for all $P \in S$ and any two $P, P' \in S$ intersect along a common face. We say a subdivision $S'$ \vocab{refines} a subdivision $S$ if every element of $S'$ is contained in an element of $S$. A \vocab{triangulation} is a subdivision consisting of simplices.

In the case $d=2$, the subdivisions of $Q$ under refinement form a lattice which is isomorphic to the face lattice of $\Sigma_Q$. For $d > 2$, the same is true, as long as we restrict ourselves to subdivisions that are \vocab{regular}: those that can be obtained by projecting the upper facets of $\conv\set{(p, \psi(p)) \in \mathbb{R}^{d+1} \mid p \in V(Q)}$ down to $\mathbb{R}^d$ for some $\psi \colon V(Q) \to \mathbb{R}$. In particular, the vertices of $\Sigma_Q$ correspond to regular triangulations. The facets of $\Sigma_Q$ correspond to regular subdivisions $S$ that are maximally coarse, in the sense that they are not the trivial subdivision $\set{Q}$ and that there is no nontrivial regular subdivision $S' \neq S$ such that $S$ refines $S'$. A vertex corresponding to a triangulation $T$ is contained in a facet corresponding to a subdivision $S$ if and only if $T$ refines $S$. As a result, $\KG(\Sigma_Q)$ can be defined as the graph where the vertices are regular triangulations of $Q$ and two vertices are adjacent if their corresponding triangulations share a common nontrivial regular refinement.

One major issue that arises is unlike the $2$-dimensional case, where every coarsest subdivision is given by division by a single diagonal, the coarsest subdivisions of an arbitrary polytope $Q$ do not admit a simple description. As a result, it appears difficult to generalize the techniques of \cref{sec:mainproof}. However, one might hope that for some special polytopes $Q$, all coarsest subdivisions are given by splitting $Q$ along a hyperplane. Such polytopes are called \vocab{totally splittable} and have been classified by Herrman and Joswig \cite{MR2639822}.

Although one might hope to now prove \cref{thm:main} for regular triangulations of all totally splittable polytopes, there is a further issue that arises. In two dimensions, two diagonals have a common refinement if and only if they don't cross (inside $Q$). However, this is no longer the case in higher dimensions, with a counterexample given by the regular crosspolytopes \cite[Ex.~3]{MR2639822}. The proofs of \cref{lem:vec,lem:circ} crucially rely on the noncrossing of diagonals, and thus seem unlikely to generalize. Indeed, even the result one might hope to prove is false. For example, in the case of the crosspolytope
\[Q = \conv \setmid{\pm e_i \in \setr^d}{i \in [d]},\]
the secondary polytope has dimension $d-1$, so it is reasonable to expect that $\chi(\KG(\Sigma_Q)) \geq d$. However, the regular triangulations are given by splitting along a choice of $d-1$ of the $d$ coordinate hyperplanes in $\setr^d$ (see \cite[Ex.~3]{MR2639822}). In particular, if $d \geq 3$, any two regular triangulations have a common split, so $\KG(\Sigma_Q)$ has no edges and thus has chromatic number $1$.

Hermann and Joswig showed that the totally splittable polytopes can be described as \vocab{joins} of simplices, (convex) polygons, prisms over simplices, and crosspolytopes. If we impose the further condition that all regular triangulations be cut out by noncrossing hyperplanes, our possibilities dwindle further. One can show that we cannot have a crosspolytope in $Q$, and that a join of multiple polytopes that are polygons or prisms over simplices is also forbidden. Joining a simplex does not affect $\Sigma_Q$, so the only higher-dimensional case remaining is when $Q$ is a prism over a simplex, which we now proceed to discuss.

Before we proceed, however, it is worth highlighting that, while our method essentially prevents us from considering joins of polytopes, taking joins does have a simple interpretation for the Kneser graph combinatorially, namely that we are taking the \emph{tensor product} of the respective Kneser graphs. Indeed, from \cite[Cor.~4.2.8]{MR2743368} one can infer that $\Sigma_{P \ast Q} \cong \Sigma_P \times \Sigma_Q$, where $P \ast Q$ denotes the join of convex polytopes $P,Q$, giving us $\KG(\Sigma_{P \ast Q}) = \KG(\Sigma_P) \times \KG(\Sigma_Q)$. 

\subsection{The permutohedron}
Let $Q = \Delta^1 \times \Delta^{n-1}$ be the prism over a $(d-1)$-dimensional simplex, where $\Delta^d$ denotes a $d$-dimensional simplex. We label the vertices of $Q$ with $1,2,\ldots,n,1',2',\ldots,n'$, where $1,2,\ldots,n$ form a simplex, $1',2',\ldots,n'$ form a simplex, and $ii'$ is an edge for all $i \in [n]$.

Any triangulation $T$ of $Q$ is regular and can be constructed in a bottom-up manner as follows:
\begin{itemize}
    \item Let $V_0 = [n]$.
    \item For $i = 1,2,\ldots,n$ in order, pick some vertex $i \in V_{i-1} \cap [n]$. Add the simplex with vertices given by $V_{i-1} \cup \set{i'}$ to the triangulation and let $V_i = (V_{i-1} \setminus \set{i}) \cup \set{i'}$.
\end{itemize}
In particular, each triangulation of $Q$ corresponds to a permutation of $[n]$ given by the order in which the $i$ are removed from $V_0$. In fact, $\Sigma_Q$ is the \vocab{permutohedron} $\Perm_{n-1}$, the polytope that is the convex hull of the permutations of $[n]$, when considered as points in $\setr^n$ \cite[Sect.~7.3.C]{GKZ}.
This is a $(n-1)$-dimensional polytope in the hyperplane $\setmid{(x_1,\ldots,x_n) \in \setr^{n}}{x_1+\cdots+x_n = \binom{n+1}{2}}$ (hence the notation $\Perm_{n-1}$).

Interestingly, both of our proofs of \cref{thm:main} generalize to show that $\chi(\KG(\Perm_{n-1})) \geq n$. To generalize \cref{lem:vec}, we may define $X_Q$ and $X_Q^*$ similarly to \cref{sec:mainproof} and construct, for each hyperplane $h$ splitting $Q$, an element $x_h \in X_Q$ that has positive coefficients on the points on $h$ and negative coefficients on the points not in $h$. It follows that if $h$ and $h'$ are two noncrossing splittings and $f_h$ is the facet corresponding to $h$, then $\gen{f_h^\perp, x_{h'}} > 0$. Alternatively, one can translate $\Perm_{n-1}$ to be a polytope in the space $\setmid{(x_1,\ldots,x_n) \in \setr^{n}}{x_1+\cdots+x_n = 0}$ and construct a bilinear form satisfing \cref{item:obtuse} of \cref{lem:polybu}, generalizing \cref{lem:circ}. In fact, this ends up being even simpler than \cref{lem:circ}, as we don't need to construct a bespoke inner product and can instead use the natural inner product on $\setr^n$.

The facets of $\Perm_{n-1}$ correspond to nonempty proper subsets $\emptyset \neq S \subsetneq [n]$, where the vertices $\sigma$ contained in a facet $f_S$ are those such that $\sigma(S) = [\abs{S}]$. Thus, two permutations $\sigma, \tau$ are connected in $\KG(\Perm_{n-1})$ if and only if there is no $k \in [n-1]$ such that $\sigma\inv([k]) = \tau\inv([k])$. As a result, coloring a permutation $\sigma$ by $\sigma\inv(1)$ is a proper $n$-coloring of $\KG(\Perm_{n-1})$, so $\chi(\KG(\Perm_{n-1})) = n$.

While this sounds interesting, it turns out there is a much simpler reason why $\chi(\KG(\Perm_{n-1})) \geq n$: there exists a clique of size $n$ in $\KG(\Perm_{n-1})$. Specifically, let $\rho$ be the permutation of $[n]$ given by $\rho(i) \equiv i+1 \pmod{n}$. If $H$ is the subgroup generated by $\rho$, we claim that the right cosets of $H$ are cliques in $\KG(\Perm_{n-1})$. Indeed, suppose there exists a permutation $\sigma$ and some $i \in [n-1]$ such that $\sigma$ and $\rho^i\sigma$ are not adjacent in $\KG(\Perm_{n-1})$. Then there exist some $k\in [n-1]$ such that $\sigma\inv([k]) = \sigma\inv(\rho^{-i}([k]))$, implying that $[k] = \rho^i([k])$, which is impossible.

The fact that $\KG(\Perm_{n-1})$ contains $(n-1)!$ disjoint copies of $K_n$ also answers a number of other questions about it. For instance, it follows immediately that the independence number is at most $(n-1)!$, and since an $n$-coloring exists the independence number is exactly $(n-1)!$. Moreover, each of these cliques is a (very) small color-critical subgraph of $\KG(\Perm_{n-1})$. Although this does not rule out the existence of other interesting color-critical subgraphs, it makes them perhaps less interesting to find.

\section{Concluding Remarks and Open Problems} \label{sec:final}
\subsection{Independence number of \texorpdfstring{\boldmath $\KG(\mT_n)$}{KG(T\_n)}} 
The work in this paper was initially motivated by the problem of determining the independence number $\alpha(\KG(\mT_n))$, i.e.\ the size of the largest intersecting family of triangulations in $\mT_n$. It is a conjecture of Kalai and Meagher (see \cite{KalaiMO}) that the best one can do is to take all triangulations containing, say, $\set{1,3}$, which are in bijection with triangulations of a convex $(n-1)$-gon. 
\begin{conj}[Kalai-Meagher] \label{Kalaiconj} For every $n \geq 3$,
\[\alpha(\KG(\mT_n))=\abs{\mT_{n-1}}=C_{n-3},\]
where $C_k = \frac{1}{k+1}\binom{2k}{k}$ denotes the $k$th Catalan number.
\end{conj}

\subsection{Triangulation structure}
\label{subsec:struct}
In our view, the statement of \cref{lem:s3}, especially when combined with \cref{rmk:nobetter}, is unexpectedly elegant, and could be a sign that the $\rsqa$ relation operates on some hidden structure within triangulations. Another result possibly hinting at deeper phenomena is the fact that star triangulations and $\set{\Delta, \Delta'}$, the very same triangulations that appear in \cref{lem:swapmost}, have additional natural properties in $\KG(\mT_n)$.

\begin{prop}
A triangulation $T \in \mT_n$ is contained in a triangle of $\KG(\mT_n)$ if and only if $T$ is a neither a star, $\Delta$, nor $\Delta'$.
\label{prop:K3inKG}
\end{prop}
\begin{proof}
For the ``if'' direction, note that if $T$ is a star centered at $i$, then any triangulation disjoint from $T$ must contain $\set{i-1, i+1}$. If $T =\Delta$, then any triangulation disjoint from $T$ needs to contain at least two diagonals in $\Delta'$. Similarly if $T =\Delta'$, then any triangulation disjoint from $T$ needs to contain at least two diagonals in $\Delta$. In any case, any two triangulations disjoint from $T$ must intersect.

For the ``only if'' direction, note that the proof of \cref{lem:swapmost} shows that there exists a copy of the triangulation $Z$ in $T$; we show by induction on $n$ that any such triangulation is contained in a triangle of $\KG(\mT_n)$. If $n = 6$, then such a triangle is formed by $T$ and its two rotations.

If $n \geq 7$, rotate $T$ so that $\set{n-1,1} \in T$ and deleting vertex $n$ from $T$ yields a triangulation $T' \in \mT_{n-1}$ containing a $Z$. By the inductive hypothesis, we know that there are two triangulations $T'_1,T'_2 \in \mT_{n-1}$ such that $T'_1$, $T'_2$, and $T'$ are pairwise disjoint. Now, if $i_1$ and $i_2$ are such that $\set{i_1, n-1, 1}$ and $\set{i_2, n-1, 1}$ are triangles in $T'_1$ and $T'_2$, respectively, let $T_1 = T'_1 \cup \set{\set{i_1, n}} \in \mT_n$ and $T_2 = T'_2 \cup \set{\set{i_2, n}} \in \mT_n$. We claim that $T$, $T_1$, and $T_2$ are pairwise disjoint. Since $T$, $T_1$, and $T_2$ are obtained from $T'$, $T'_1$, and $T'_2$ by adding $\set{n-1,1}$, $\set{i_1, n}$, and $\set{i_2, n}$, respectively, the only way for them to intersect is if $i_1 = i_2$. However, this contradicts the fact that $T'_1$ and $T'_2$ are disjoint.
\end{proof}

\subsection{Kneser hypergraphs} For any integer $r \geq 2$ and an arbitrary set system $\calf$, let $\KG^{r}(\calf)$ be the \vocab{Kneser $r$-uniform hypergraph} whose vertices are the elements of $\calf$ and whose edges are the collections of $r$ pairwise disjoint sets in $\calf$. In \cite{alon1986chromatic}, Alon, Frankl and Lov\'asz generalized the Lov\'asz-Kneser theorem, showing that
\[\chi\paren*{\KG^{r}\paren*{\binom{[n]}{k}}} = \ceil*{ \frac{n - r(k-1)}{r-1}}\]
for any integers $k \geq 1$ and $r \geq 2$ with $n \geq rk$. Generalizing both this fact and \cref{Dol}, Kříž \cite{kriz1992} later showed that for any set system $\calf$,
\[\chi(\KG^{r}(\calf)) \geq \ceil*{ \frac{1}{r-1} \cd_r(\calf)}.\]
Here $\cd_{r}(\calf)$ denotes the \vocab{$r$-colorability defect} of $\calf$, the minimum size of a set $S'$ such that the elements of $\calf$ disjoint from $S'$ form an $r$-colorable family. Subsequently, stability versions of these results and further generalizations and refinements have been studied \cite{ziegler2002generalized,alon2009stable,frick2020chromatic}. It is therefore a natural direction to next investigate the appropriate hypergraph generalizations of \cref{thm:main}, \cref{thm:stab}, and \cref{thm:star}.

For every $r \geq 2$, similar arguments to those in \cref{sec:intro} show that $\chi(\KG^{r}(\mT_n)) \leq \ceil*{ {(n-2)}/({r-1})}$. Surprisingly, this is not optimal for $r \geq 3$. Indeed, by \cref{prop:K3inKG}, no star is contained in an edge of $\KG^{r}(\mT_n)$, so one can obtain 
\[\chi(\KG^{r}(\mT_n)) \leq \ceil*{ \frac{n-3}{r-1} }\]
by extending \cref{lem:starconstruct}. 
Although equality does not hold for all $n \geq 3$, this suggests the following conjecture:
\begin{conj}
For any $r \geq 2$, there exists some $n_0(r)$ such that for all $n \geq n_0(r)$ we have
\[\chi(\KG^{r}(\mT_n)) = \ceil*{ \frac{n-r}{r-1} }.\]
\end{conj}
It seems likely that any proof or disproof will in turn shed light on the structure of $\mT_n$ as hinted in \cref{subsec:struct}. Note that, as with \cref{Dol}, Kříž's theorem only gives trivial bounds for $\KG^{r} (\mT_n)$. 

\section*{Acknowledgments}
We thank Santiago Morales for fruitful discussions. C.P.\ was supported by NSF grant DMS-2246659. D.Z.\ was supported by the NSF Graduate
Research Fellowship Program (grant DGE-2039656).

\printbibliography
\end{document}